\documentclass[11pt]{amsart}

\usepackage{latexsym,amssymb,amsthm,amsmath,amscd}
\theoremstyle{plain}  
\newtheorem{theorem}{Theorem}[section]

\newtheorem{lemma}{Lemma}[section]
\newtheorem{proposition}{Proposition}[section]

\newtheorem{corollary}{Corollary}[section]

\numberwithin{equation}{section}

\theoremstyle{remark}
\newtheorem{remark}{Remark}[section]

 \numberwithin{equation}{section}

 \numberwithin{equation}{section}
\def\<{\left < }
\def\>{\right >}
\def\({\left ( }
\def\){\right )}

\def\sech{\,{\rm sech\,}}
\def\e{\eqref}
\def\i{{\rm i}\hskip.01in}
\begin{document}
\def\e{\eqref}

\title[Minimal flat Lorentzian surfaces]
{Minimal flat Lorentzian surfaces in Lorentzian complex space forms}

\author[B. Y. Chen]{BANG-YEN CHEN (East Lansing)}

 \address{Department of Mathematics, 
	Michigan State University, East Lansing, MI 48824--1027, USA}

\email{bychen@math.msu.edu}

\begin{abstract}   In this article we study  minimal flat Lorentzian surfaces in Lorentzian complex space forms. First we prove that,  for minimal flat Lorentzian surfaces in a Lorentzian complex form, the equation of Ricci is a consequence of the equations of Gauss and Codazzi. Then we classify minimal flat Lorentzian surfaces in the Lorentzian complex plane ${\bf C}^2_1$. Finally, we classify minimal flat slant surfaces in  Lorentzian complex projective plane $CP^2_1$ and in Lorentzian complex hyperbolic plane $CH^2_1$.
\end{abstract}

\keywords{Lorentzian surface; slant surfaces; minimal surface; Lagrangian surface; Lorentzian complex  space form.}

 \subjclass[2000]{Primary: 53C40; Secondary  53C42, 53C50}

\date{}
 
\maketitle

\section{Introduction}

Let $\tilde M^n_i(4c)$ be an indefinite complex space form of complex dimension $n$ and {\it complex index} $i$. The complex index is defined as the complex dimension of the largest complex negative
definite subspace of the tangent space. If $i=1$, we say that $\tilde M^n_i(4c)$ is Lorentzian. The curvature tensor $\tilde R$ of $\tilde M^n_i(4c)$ is given by
\begin{equation}\begin{aligned}\label{1.1}&\tilde R(X,Y)Z=c\{\left<Y,Z\right> \hskip-.02in X\hskip-.02in - \hskip-.02in \left<X,Z\right>\hskip-.02in  Y\hskip-.02in +\hskip-.02in \left<JY,Z\right> JX\\& \ \hskip.9in  -\left<JX,Z\right>\hskip-.02in JY\hskip-.02in +\hskip-.02in 2\left<X,JY\right>\hskip-.02in  JZ\}.\end{aligned}\end{equation}

 Let {\bf C}$^n$ denote the  complex number $n$-space with complex coordinates
$z_1,\ldots,$ $z_n$. The {\bf C}$^n$ endowed with $g_{i,n}$, i.e., the real part of the Hermitian form
$$b_{i,n}(z,w)=-\sum_{k=1}^i \bar z_kw_k +\sum_{j=i+1}^n \bar z_jw_j,\quad z,w\in\hbox{\bf C}^n,$$ defines a flat indefinite complex space form with complex index $i$. We simply denote the pair ({\bf C}$^n,g_{i,n})$ by {\bf C}$^n_i$. 
Consider the differentiable manifold:
$$S^{2n+1}_2(c)=\{z\in \hbox{\bf C}_1^{n+1} \,;\, b_{1,n+1}(z,z)= c^{-1}>0 \},$$ which is an indefinite real space form of constant sectional curvature $c$.   The Hopf fibration
$$\pi: S^{2n+1}_2(c)\to CP^n_1(4c):
z\mapsto z\cdot\hbox{\bf C}^*$$ is a submersion and there exists a unique pseudo-Riemannian
metric of complex index one on $CP^n_1(4c)$ such that
$\pi$ is a Riemannian submersion. The pseudo-Riemannian manifold $CP^n_1(4c)$ is a
Lorentzian complex space form of positive holomorphic sectional curvature $4c$.

Analogously, if $c<0$, consider
$$H^{2n+1}_2(c)=\{z\in\hbox{\bf C}_2^{n+1}\,;\, b_{2,n+1}(z,z)= c^{-1}<0 \},$$ which is an
indefinite real space form of constant sectional curvature $c<0$. The Hopf fibration
$$\pi: H^{2n+1}_2(c)\to CH^n_1(4c): z\mapsto z\cdot\hbox{\bf C}^*$$ is a submersion
and there exists a unique pseudo-Riemannian metric of complex index 1 on $CH^n_1(4c)$
such that $\pi$ is a Riemannian submersion. The pseudo-Riemannian manifold $CH^n_1(4c)$ is a Lorentzian complex space form of negative
holomorphic sectional curvature $4c$.

A complete simply-connected Lorentzian complex
space form  $\tilde M^n_1(4c)$ is holomorphically isometric to {\bf C}$^n_1$,  $CP^n_1(4c)$, or $CH^n_1(4c)$, according
to $c=0, c>0$ or $c<0$, respectively.

 Lorentzian surfaces in pseudo-Riemannian spaces of constant curvature with signature $(2,2)$ have been studied by L. Verstraelen and M. Pieters \cite{VP1,VP2} among others. In this article, we study minimal flat Lorentzian surfaces in Lorentzian complex space forms.

In section 3 of this article, we provide the basic results  for Lorentzian surfaces in Lorentzian K\"ahler surfaces. In particular, we show that each tangent plane of   a Lorentzian surface cannot be $J$-invariant. In section 4, we prove that  the equation of Ricci is a consequence of equations of Gauss and Codazzi for  minimal flat Lorentzian surfaces  in  Lorentzian complex space forms.
The complete classification of minimal flat Lorentzian surfaces in Lorentzian complex  plane ${\bf C}^2_1$ is obtained in section 5. In section 6 we show that the only minimal flat slant surfaces in non-flat Lorentzian complex space forms are the Lagrangian ones.
In this section, we also classify minimal flat slant surfaces in Lorentzian complex plane $CP^2_1$. In the last section, we provide the classification of minimal flat Lagrangian surfaces in the Lorentzian complex hyperbolic plane $CH^2_1$.

\section{Preliminaries}

Let $M$ be a Lorentzian surface of a Lorentzian K\"ahler surface $\tilde M^2_1$ equipped with an almost complex structure $J$ and metric $\tilde g$. Let $\<\;\,,\;\>$ denote the inner product associated with $\tilde g$.
Denote  the induced metric on $M$ by $g$. 

Let  $\nabla$ and $\tilde\nabla$ denote the Levi-Civita connection on $M$ and $\tilde M^2_1$, respectively. Then the formulas of Gauss and Weingarten are given respectively by (cf. \cite{c0,c1,CD,CV})
\begin{align} &\label{2.1}\tilde \nabla_XY=\nabla_XY+h(X,Y),\\& \label{2.2}\tilde\nabla_X\xi=-A_\xi X+D_X \xi\end{align}
for vector fields $X,Y$ tangent to $M$ and $\xi$ normal to $M$,
where $h,A$ and $D$ are the second fundamental form, the shape operator and the normal connection, respectively.

The shape operator and the second fundamental form are related by
\begin{align}\label{2.3} \<h(X,Y),\xi\>=\<A_{\xi}X,Y\>\end{align}
for $X,Y$ tangent to $M$ and $\xi$ normal to $M$. 

For each normal vector $\xi$ of $M$ at $x\in M$, the shape operator $A_{\xi}$ is a symmetric endomorphism of the tangent space $T_xM$. However, for Lorentzian surfaces the shape operator $A_\xi$ is not diagonalizable in general.

The mean curvature vector  is defined by \begin{align}\label{2.4} &H=\frac{1}{2}{\rm trace}\, h.\end{align}
A Lorentzian surface in $\tilde M^2_1$ is called {\it minimal}  if $H=0$ at each point on $M$. 

For a Lorentzian surface $M$ in a Lorentzian complex space form $\tilde M^2_1(4c)$, the equations of Gauss, Codazzi and Ricci are given respectively by
\begin{align}\label{2.5} & \<R(X,Y)Z,W\> =\<\right.\hskip-.02in\tilde R(X,Y)Z,W\left.\hskip-.02in \>+  \<h(X,W),h(Y,Z)\>\\&\notag \hskip1.2in  - \<h(X,Z),h(Y,W)\>, \\ &  \label{2.6} (\tilde R(X,Y)Z)^\perp=(\bar\nabla_X h)(Y,Z) - (\bar\nabla_Y h)(X,Z),
\\&\label{2.7} \<R^D(X,Y)\xi,\eta\>=\<\right.\tilde R(X,Y)\xi,\eta\left.\>+\<[A_\xi,A_\eta]X,Y\>,
\end{align}
where $X,Y,Z,W$ are vector tangent to $M$,
and $\nabla h$ is defined by
\begin{equation}\begin{aligned}\label{2.8}(\bar \nabla_X h)(Y,Z) = D_X h(Y,Z) - h(\nabla_X Y,Z) - h(Y,\nabla_X
Z).\end{aligned}\end{equation}

\section{Basics on  Lorentzian   surfaces}

Let $M$ be a Lorentzian surface in a Lorentzian K\"ahler surface $(\tilde M^2_1,g,J)$. For each tangent vector $X$ of $M$, we put
\begin{align} \label{3.1} &JX=PX+FX ,\end{align}
where $PX$ and $FX$ are the tangential and the normal components of $JX$. 

On the Lorentzian surface $M$ there exists a {\it pseudo-orthonormal}  local frame $\{ e_1, e_2\}$ on $M$ such that 
\begin{align} \label{3.2} & \<e_1, e_1\>=\< e_2,e_2\>=0,\;\<e_1, e_2\>=-1.\end{align} 
For a pseudo-orthonormal frame $\{e_1,e_2\}$ on  $M$  satisfying \e{3.2}, it follows from \e{3.1},  \e{3.2}, and $\<JX,JY\>=\<X,Y\>$  that 
\begin{align} \label{3.3} &Pe_1=(\sinh\alpha) e_1,\;\; Pe_2=-(\sinh \alpha) e_2\end{align} for some function $\alpha$, which is called the Wirtinger angle. 

When the Wirtinger angle $\alpha$ is constant on $M$, the Lorentzian  surface $M$ is called a {\it slant surface} (cf. \cite{c1,CD,CT}).  In this case, $\alpha$ is called the {\it slant angle}; the slant surface is called $\alpha$-slant. A $\alpha$-slant  surface is Lagrangian if and only if $\alpha=0$.  
Obviously, slant surfaces (in particular, Lagrangian surfaces) in a Lorentzian K\"ahler surface are Lorentzian surfaces.

If we put 
\begin{align} \label{3.4} &e_3=(\sech \alpha)Fe_1,\;\; e_4=(\sech \alpha)Fe_2,\end{align}
then we find from \e{3.1}-\e{3.4}  that
 \begin{align}& \label{3.5}J e_1=\sinh \alpha e_1+\cosh\alpha e_3,\hskip.2in Je_2=-\sinh\alpha e_2+\cosh\alpha e_4,\\& \label{3.6}Je_3=-\cosh \alpha e_1-\sinh \alpha e_3,\;\;  Je_4=-\cosh \alpha e_2+\sinh \alpha e_4,\\\label{3.7} &\<e_3,e_3\>=\<e_4,e_4\>=0,\;\; \<e_3,e_4\>=-1.\end{align}
   We call such a frame $\{e_1,e_2,e_3,e_4\}$ chosen above an {\it adapted pseudo-orthonormal frame} for the Lorentzian  surface $M$  in $\tilde M^2_1$.

 From \e{3.5} we obtain the following.
 
 \begin{proposition} Let $M$ be a Lorentzian surface in Lorentzian K\"ahler surface $(\tilde M^2_1,g,J)$. Then every tangent plane of $M$ is not $J$-invariant.
 \end{proposition}

 We need the following.

 \begin{lemma}  \label{L:3.1}  If $M$ is a Lorentzian surface in a Lorentzian K\"ahler surface $\tilde M^2_1$, then with respect to  an adapted pseudo-orthonormal  frame  we have
\begin{align} &\label{3.8} \nabla_X e_1=\omega(X) e_1,\;\;  \nabla_X e_2=-\omega(X) e_2, \\& \label{3.9}
 D_X e_3=\Phi(X) e_3,\; \; D_X e_4=-\Phi(X) e_4 \end{align}
for some 1-forms $\omega,\Phi$ on $M$.
\end{lemma}
\begin{proof} Let us put
\begin{align} \label{3.10} &\nabla_X e_1=\omega_1^1(X)e_1+\omega_1^2(X)e_2,\;\; \nabla_X e_2=\omega_2^1(X)e_1+\omega_2^2(X)e_2.
\end{align} Then we obtain from \eqref{3.2} that $\omega_1^2=\omega_2^1=0$ and $ \omega_2^2=-\omega_1^1.$ Thus, if we put $\omega=\omega_1^1$, then we get \eqref{3.8}.
Similarly, if we put
\begin{align} \label{3.11} &  D_Xe_3=\omega_3^3(X)e_3+\omega_3^4(X)e_4,\;\; D_Xe_4= \omega_4^3(X)e_3+\omega_4^4 (X)e_4,\end{align} then it follows from \eqref{3.7} that $\omega_3^4=\omega_4^3=0$ and $ \omega_3^3=-\omega_4^4.$ So, after putting $\Phi=\omega_3^3$, we get \eqref{3.9}. \end{proof}

For a Lorentzian surface $M$ in $\tilde M^2_1$ with second fundamental form $h$, we put \begin{align}\label{3.12} h(e_i,e_j)=h^3_{ij}e_3+h^4_{ij}e_4,\end{align}
where $e_1,e_2,e_3,e_4$ is an adapted pseudo-orthonormal frame.

\begin{lemma}  \label{L:3.2}  If $M$ is a Lorentzian surface in a Lorentzian K\"ahler surface $\tilde M^2_1$, then with respect to  an adapted pseudo-orthonormal frame  $\{e_1,e_2,e_3,e_4\}$ we have
\begin{align} & \label{3.13} \begin{cases} A_{e_3}e_j=h^4_{j2}e_1+h^4_{1j}e_2,\\ A_{e_4}e_j=h^3_{j2}e_1+h^3_{1j}e_2,\end{cases}\\&\label{3.14}e_j\alpha=(\omega_j- \Phi_j)\coth \alpha -2h^3_{1j},
\\ \label{3.15} &e_1\alpha=h^4_{12}-h^3_{11},\;\; e_2\alpha=h^4_{22}-h^3_{12},
\\&\label{3.16} \omega_j-\Phi_j=(h^3_{1j}+h^4_{j2})\tanh\alpha,
\end{align} for $j=1,2$, where $\omega_j=\omega(e_j)$ and $\Phi_j=\Phi(e_j)$.
\end{lemma}
\begin{proof} This is done by direct computation using $\tilde\nabla_X(JY)=J\tilde\nabla_XY$ together with \e{3.5}-\e{3.7}, and Lemma 3.2.
\end{proof}

\section{Fundamental equations of minimal flat Lorentzian surfaces}

In general, the three fundamental equations of Gauss, Codazzi and Ricci are independent. However, for  minimal flat Lorentzian surfaces  in Lorentzian complex space forms we have the following.

\begin{theorem}\label{P:4.1} The equation of Ricci is a consequence of the equations of Gauss and Codazzi for  minimal flat  Lorentzian surfaces  in a Lorentzian complex space form $\tilde M_1^2(4c)$.
\end{theorem}
\begin{proof} Let $M$ be a  minimal flat Lorentzian  surface in  a Lorentzian complex space form $\tilde M^2_1(4c)$. Since $M$ is flat, we may assume that $M$ is  an open connected subset  of  $\mathbb E^2_1$ equipped with the Lorentzian metric tensor:
\begin{align}\label{4.1} g_o=-dx\otimes dy-dy\otimes dx.\end{align}

Put $e_1=\partial/\partial x,e_2=\partial/\partial y$. Then $\{e_1,e_2\}$ is a pseudo-orthonormal frame on $M$ such that $\nabla e_1=\nabla e_2=0$. Thus, we have $\omega=0$.

 Let $e_3,e_4$ be the normal vector fields as \e{3.4}. Then $\{e_1,e_2,e_3,e_4\}$ is an adapted pseudo-orthonormal frame. Since $M$ is  minimal and Lorentzian,
it follows from  \e{2.4} and \e{3.2} that 
\begin{align} \label{4.2} & h(e_1,e_1)=\beta e_3+\gamma e_4,\; h(e_1,e_2)=0,\;\; h(e_2,e_2)=\lambda e_3+\mu e_4
\end{align} for some functions $\beta,\gamma,\lambda,\mu$. 

After applying Lemma \ref{L:3.2}  we find from \e{4.2} that 
\begin{equation}\begin{aligned}\label{4.3}& (\bar\nabla_{e_1}h)(e_1,e_2)=(\bar\nabla_{e_2}h)(e_1,e_2)=0,
\\& (\bar\nabla_{e_2}h)(e_1,e_1)=(\beta_y+\beta \Phi_2)e_3+(\gamma_y-\gamma \Phi_2)e_4,\\&   (\bar\nabla_{e_1}h)(e_2,e_2)=(\lambda_x+\lambda \Phi_1)e_3+(\mu_x-\mu \Phi_1)e_4.
\end{aligned}\end{equation}
On the other hand, it follows from \e{1.1} and \e{3.5} that
\begin{equation}\begin{aligned}\label{4.4}& (\tilde R(e_1,e_2)e_1)^\perp=3 c\sinh\alpha\cosh\alpha e_3,\\&  (\tilde R(e_1,e_2)e_2)^\perp=3 c\sinh \alpha\cosh\alpha e_4.
\end{aligned}\end{equation} 
Thus, by using \e{4.3}, \e{4.4}, we obtain from the equation of Codazzi  that
\begin{align}\label{4.5}& \beta_y=-\beta \Phi_2-3 c\sinh\alpha\cosh\alpha,\\&\label{4.6}
\gamma_y=\gamma \Phi_2,\;   \lambda_x=-\lambda \Phi_1,\\& \label{4.7} \mu_x=\mu \Phi_1+3 c\sinh\alpha\cosh\alpha.
\end{align}
Also, it follows from \e{4.2}, $\omega=0$ and Lemma \ref{L:3.2} that
\begin{align} & \label{4.8} A_{e_3}e_1=\gamma e_2,\; A_{e_3}e_2=\mu e_1,\; 
  A_{e_4}e_1=\beta e_2,\; A_{e_4}e_2=\lambda e_1,\\ \label{4.9} &\beta=-\alpha_x ,\;\; \mu=\alpha_y, 
\\&\label{4.10} \Phi_1=\alpha_x \tanh\alpha,\; \Phi_2=-\alpha_y \tanh\alpha.
\end{align} 
Substituting \e{4.9} and \e{4.10} into \e{4.5} and \e{4.7} gives
\begin{align}\label{4.11}& \alpha_{xy}=\alpha_x\alpha_y\tanh\alpha+3 c\sinh\alpha\cosh\alpha.\end{align} 

In views of \e{1.1}, \e{3.5}-\e{3.7}, \e{4.2}, and \e{4.8}, the equation of Gauss becomes \begin{align}\label{4.12}\gamma\lambda=\alpha_x\alpha_y+c(3\sinh^2\alpha-1).\end{align}

On the other hand, by applying \e{3.5} and \e{3.6}, we have
\begin{align}\label{4.13} \<\right.\hskip-.02in  \tilde R(e_1,e_2)e_3,e_4\hskip-.02in \left.\>=c(3\sinh^2\alpha+1).\end{align}
Using $\omega=0$, Lemma \ref{L:3.1} and \e{4.8}-\e{4.10}, we find
\begin{align}\label{4.14} \<R^D(e_1,e_2)e_3,e_4\>&=e_2\Phi_1-e_1\Phi_2\\& \notag=2\alpha_{xy}\tanh\alpha+2\alpha_x\alpha_y \sech^2\alpha,\\ \label{4.15} \<[A_{e_3},A_{e_4}]e_1,e_2\>&=\gamma \lambda+\alpha_x\alpha_y.\end{align}
Hence, in view of \e{4.9}, \e{4.13}, \e{4.14} and \e{4.15}, the equation of Ricci becomes
\begin{align}\label{4.16}2\alpha_{xy}\tanh\alpha+2\alpha_x\alpha_y \sech^2\alpha=\gamma\lambda=\alpha_x\alpha_y+c(3\sinh^2\alpha-1).\end{align} After applying \e{4.11}, the equation \e{4.16} of Ricci  can be simplified exactly as the equation \e{4.12} of Gauss.
\end{proof}

\section{Classification of minimal flat Lorentzian surfaces in ${\bf C}^2_1$}

Minimal flat Lagrangian surfaces in the Lorentzian complex plane ${\bf C}^2_1$ have been classified by B. Y. Chen and L.Vrancken in \cite{CV}. Clearly, Lagrangian surfaces in ${\bf C}^2_1$ are Lorentzian surfaces automatically. 
In this section we completely classify minimal flat Lorentzian surfaces in the Lorentzian complex plane ${\bf C}^2_1$.

\begin{theorem}  \label{T:5.1}  Let $\alpha(y)$ and $f(y)$ be two arbitrary differentiable functions of single variable defined on an open interval $I\ni 0$. Then
 \begin{equation}\begin{aligned}\notag & \psi(x,y)=\Bigg(x+\i f(y)+\frac{1}{2} \int_0^y \cosh^2\alpha dy-\int_0^y f'(y)\sinh\alpha dy,\\&\hskip.1in  x-y+\i f(y) +\frac{1}{2} \int_0^y \cosh^2\alpha dy-\int_0^y f'(y)\sinh\alpha dy-\i \int_0^y \sinh\alpha dy\Bigg)\end{aligned}\end{equation}  defines a minimal flat Lorentzian surface in  the Lorentzian complex plane ${\bf C}^2_1$ with $\alpha$ as its Wirtinger angle.

Conversely, every  minimal flat  Lorentzian surface in ${\bf C}^2_1$ is either an open portion of a totally geodesic Lorentzian plane or congruent to the Lorentzian surface described above. 
\end{theorem}
\begin{proof} It is straight-forward to show that the mapping $\psi$ defined in the theorem gives rise to a minimal flat Lorentzian surface in ${\bf C}^2_1$.

Conversely, assume that $M$ is a minimal flat Lorentzian  surface in ${\bf C}^2_1$. If the second fundamental form  vanishes identically, then $M$ is an open portion of a totally geodesic  Lorentzian  plane. 
So, we assume from now on that $M$ is a non-totally geodesic minimal flat Lorentzian surface in ${\bf C}^2_1$. 

Since $M$ is flat, we may assume that as before that $M$ is  an open connected subset  of  $\mathbb E^2_1$ equipped with the Lorentzian metric tensor:
\begin{align}\label{5.1} g_o=-dx\otimes dy-dy\otimes dx.\end{align}

Put $e_1=\partial/\partial x,e_2=\partial/\partial y$. Then $\{e_1,e_2\}$ is a pseudo-orthonormal frame on $M$ such that $\nabla e_1=\nabla e_2=0$. Thus, we have $\omega=0$.

 Let $e_3,e_4$ be the normal vector fields defined by \e{3.4}.  Since $M$ is a minimal Lorentzian surface,
we have
\begin{align} \label{5.2} & h(e_1,e_1)=\beta e_3+\gamma e_4,\; h(e_1,e_2)=0,\;\; h(e_2,e_2)=\lambda e_3+\mu e_4,
\end{align} for some functions $\beta,\gamma,\lambda,\mu$. By applying \e{3.7}, \e{5.2} and the equation of Gauss, we find \begin{align}\label{5.3}\gamma\lambda=-\beta \mu.\end{align}
\vskip.05in

{\it Case} (A): $\beta=0$ {\it on} $M$. From \e{5.3}, we get $\gamma\lambda=0$.
\vskip.05in

{\it Case} (A.1): $\gamma=0$  {\it on} $M$. In this case, \e{5.2} reduces to
 \begin{align} \label{5.4} & h(e_1,e_1)= h(e_1,e_2)=0,\;\; h(e_2,e_2)=\lambda e_3+\mu e_4.
\end{align} 
Since $M$ is not totally geodesic, at least one of  $\lambda,\mu$ is a nonzero function. 
Now, by applying the equation of Codazzi, we find from \e{5.4} that
\begin{align} \label{5.5} & \lambda_x=-\lambda \Phi_1,\;\; \mu_x=\mu \Phi_1.
\end{align}

On the other hand, it follows from $\omega=0$, \e{5.4}, and Lemma \ref{L:3.2} that \begin{align} \label{5.6} &\alpha_x=0,\;\; \alpha_y=\mu=-\Phi_2 \coth\alpha ,\;\;   \Phi_1=0.\end{align}
From the first two equations in \e{5.6}, we get $\alpha=\alpha(y)$ and $\mu=\alpha'(y)$. Also, from \e{5.5} and the last equation in \e{5.6}, we have  $\lambda=\lambda(y)$ and $\mu=\mu(y)$.
Therefore, after applying \e{3.5}, \e{5.4} and the formula of Gauss, we know that the immersion of the surface in ${\bf C}^2_1$ satisfies
\begin{equation}\begin{aligned} \label{5.7} & \psi_{xx}=\psi_{xy}=0,\\& \psi_{yy}=\lambda(y)(\i \sech\alpha-\tanh\alpha)\psi_x+\alpha''(y) (\i \sech\alpha+\tanh\alpha)\psi_y.
\end{aligned} \end{equation}
Solving the first two equations of  \e{5.7}  shows that the immersion is given by
\begin{align} \label{5.8} & \psi=c_1 x+B(y)
\end{align} 
for some vector $c_1\in {\bf C}^2_1$ and ${\bf C}^2_1$-valued function $B(y)$. Thus, by applying  \e{5.1} and $\<\i \psi_x,\psi_y\>=-\sinh\alpha$, we may  find from \e{3.2} and \e{3.5} that
\begin{align} \label{5.9} & \<c_1,c_1\>=0,\;\; \<c_1,B'\>=-1,\;\; \<\i c_1, B'\>=-\sinh\alpha,\\&\label{15.10} \<B',B'\>=0.
\end{align} 
Without loss of generality, we may put \begin{align} \label{5.10} &c_1=(1,1),\;\; B(y)=(k(y)+\i f(y), u(y)+\i v(y)).\end{align} Now, by applying conditions in \e{5.9} and \e{5.10}, we obtain 
\begin{equation}\begin{aligned}&\label{15.12} u=k-y+a_1,\;\; v=f-\int_0^y \sinh\alpha dy+a_2\end{aligned} \end{equation}
for some real numbers $a_1,a_2$. From \e{15.10} and \e{15.10}, we find
\begin{equation}\begin{aligned}&\label{15.13}  k=\frac{1}{2}\int_0^y\cosh^2\alpha dy-\int_0^y f'(y)\sinh\alpha dy+a_3\end{aligned} \end{equation} for some real number $a_3$.

By combining \e{5.8}, \e{5.10}, \e{15.12} and \e{15.13} we know that  the immersion is congruent to the one described in the theorem.

\vskip.05in

{\it Case} (A.2): $\lambda=0$  {\it and} $\gamma\ne 0$ {\it on some open subset $U\subset M$}. Let us work on $U$. From \e{5.2} we have\begin{align} \label{5.11} & h(e_1,e_1)=\gamma e_4,\; h(e_1,e_2)=0,\;\; h(e_2,e_2)=\mu e_4.
\end{align} Thus, the equation of Codazzi yields
\begin{align} \label{5.12} & \gamma_y=\gamma \Phi_2,\;\; \mu_x=\mu \Phi_1.\end{align}

When $\mu=0$, this  reduces to case (A.1) after interchanging $x$ and $y$. So, we assume that $\mu\ne 0$. Hence \e{5.12} gives \begin{align} \label{5.13} & (\ln\gamma)_y= \Phi_2,\;\; (\ln \mu)_x= \Phi_1.\end{align}
It follows from \e{3.9} of Lemma \ref{3.2} that the normal curvature tensor $R^D$ satisfies
\begin{equation}\begin{aligned} \label{5.14}   \<R^D(e_1,e_2)e_3,e_4\>&=\<D_{e_1}(\Phi_2 e_3)-D_{e_2}(\Phi_1 e_3),e_4\>\\&=e_2\Phi_1-e_1\Phi_2\\&=(\ln \mu)_{xy}-(\ln \gamma)_{xy}.\end{aligned} \end{equation}

On the other hand,  from \e{3.13} of Lemma \ref{L:3.2} and \e{5.11} we get $A_{e_4}=0$. Thus, by combining these with the equation of Ricci, we obtain
$(\ln \gamma)_{xy}=(\ln \mu)_{xy}.$ Consequently, we have
\begin{align} \label{5.15} & \gamma= (f(x)+k(y))\mu\end{align}
for some real-valued functions $f(x),k(y)$. Therefore, after applying \e{3.5}, \e{5.11} and the formula of Gauss, we know that the immersion  satisfies
\begin{equation}\begin{aligned} \label{5.16} & \psi_{xx}=(f(x)+k(y))\mu (\i \sech\alpha-\tanh\alpha)\psi_y, \;\;  \\& \psi_{xy}=0,\\& \psi_{yy}=\mu (\i \sech\alpha-\tanh\alpha)\psi_y.
\end{aligned} \end{equation}
It follows from $(\psi_{yy})_x=(\psi_{xy})_y=0$ that $$\mu_x=-\i \mu \alpha_x \sech\alpha.$$ Hence, we have
\begin{align} \label{5.17} & \mu=\phi (y)e^{-2\i \tan^{-1}(\tanh \alpha/2)}\end{align}
for some nonzero real-valued function $\phi(y)$. Substituting this into \e{5.15} gives
\begin{equation}\begin{aligned} \label{5.18} & \psi_{xx}=\i \phi(y) (f(x)+k(y))\psi_y, \;\;  \\& \psi_{xy}=0,\;\;  \psi_{yy}=\i \phi(y)\psi_y.\end{aligned} \end{equation}
Now, it follows from $(\psi_{xx})_y=(\psi_{xy})_x=0$ and \e{5.18} that
\begin{align} \label{5.19} & \i[ (\phi(y) k'(y)+(f(x)+k(y))\phi'(y)]=(f(x)+k(y))\phi^2(y).\end{align}
Since $\phi, f, k$ are real-valued, \e{5.19} implies that $(f(x)+k(y))\phi(y)$=0. But this is impossible, since $\gamma$ and $\mu$ are nonzero functions. Thus, this case cannot occur.

\vskip.05in

{\it Case} (B): $\gamma=0$ {\it and} $\beta\ne 0$  {\it on some open subset $V\subset M$}. Let us work on $V$. It follows from \e{5.3} that $\mu=0$. Hence, \e{5.2} reduces to 
\begin{align} \label{5.20} & h(e_1,e_1)=\beta e_3,\; h(e_1,e_2)=0,\;\; h(e_2,e_2)=\lambda e_3.
\end{align}
But this case is also impossible after applying a similar argument as case (A.2).

\vskip.05in

{\it Case} (C): $\beta ,\gamma,\lambda,$ {\it and} $\mu$ {\it are nonzero}  {\it on some open subset $W\subset M$}. Let us work on $W$. It follows from \e{5.2}, $\omega=0$, Lemma \ref{L:3.1}, and the equation of Codazzi that
\begin{align} \label{5.21} & (\ln \beta)_y=-\Phi_2,\; (\ln\gamma)_y=\Phi_2,\;  (\ln \lambda)_x=-\Phi_1,\; (\ln \mu)_x=\Phi_1,\end{align}
which imply that \begin{align} \label{5.22} &  \beta\gamma=\varphi(x),\;\; \lambda\mu=\eta(y)\end{align}
for some nonzero real-valued functions $\varphi(x),\eta(y)$. Hence, \e{5.2} becomes 
\begin{equation}\begin{aligned}  \label{5.23} & h(e_1,e_1)=\beta e_3+\frac{\varphi(x)}{\beta} e_4,\; \\& h(e_1,e_2)=0,\;\; \\& h(e_2,e_2)=\frac{\eta(y)}{\mu}e_3+\mu e_4.\end{aligned} \end{equation}
Since the surface is flat, \e{5.23} and the equation of Gauss gives
\begin{align} \label{5.24} & \beta^2\mu^2=-\varphi(x)\eta(y).\end{align}
By applying \e{3.5}, \e{5.23} and the formula of Gauss, we know that the immersion satisfies
\begin{equation}\begin{aligned} \label{5.25} & \psi_{xx}=\beta (\i \sech\alpha-\tanh\alpha)\psi_x+\frac{\varphi(x)}{\beta} (\i \sech\alpha+\tanh\alpha)\psi_y , \;\;  \\& \psi_{xy}=0,\\& \psi_{yy}=\frac{\eta(y)}{\mu} (\i \sech\alpha-\tanh\alpha)\psi_x+\mu (\i \sech\alpha+\tanh\alpha)\psi_y.
\end{aligned} \end{equation}

The compatibility conditions of  system \e{5.25} are given by
\begin{align} \label{5.26} &\mu \beta \beta_y=- \varphi(x)\eta(y) \tanh\alpha,
\\& \label{5.27}\mu=\alpha_y,
\\& \label{5.28} \beta^2\mu \alpha_y=-\varphi(x)\eta(y),
\\& \label{5.29} \beta_y=\beta\mu \tanh\alpha,
\\& \label{5.30}(\sinh2\alpha) \mu_x-2\mu \alpha_x=(3-\cosh 2\alpha)\beta\mu,
\\& \label{5.31}\mu_x+\mu \alpha_x\tanh\alpha=-2\beta\mu\tanh\alpha,
\\& \label{5.32}\beta\mu ^2\alpha_x=\varphi(x)\eta(y),
\\& \label{5.33}\beta \mu \mu_x=\varphi(x)\eta(y)\tanh\alpha
.\end{align}

Form  \e{5.27}, \e{5.28}, and \e{5.32}, we get 
\begin{align} \label{5.34} &\beta=-\alpha_x,\;\; \mu=\alpha_y.\end{align}
Thus,  \e{5.28}, \e{5.31} and \e{5.34} imply that
\begin{align} & \label{5.35} \alpha_x^2 \alpha_y^2=-\varphi(x)\eta(y),
\\& \label{5.36}\alpha_{x y} =\alpha_x\alpha_y  \tanh\alpha.\end{align}

Solving \e{5.36} yields
\begin{align} \label{5.37} &\alpha=2 \tanh^{-1}(\tan(f(x)+k(y)))\end{align}
for some functions $f(x),k(y)$. Since $\varphi(x)\eta(y)\ne 0$, \e{5.35} shows that $\alpha$ is a non-constant function. Hence, $f(x)+k(y)$ is also non-constant.

Substituting \e{5.37} into \e{5.35} gives
\begin{align} \label{5.38} &16 \text{\small$ \(\frac{f'(x)^2}{\varphi(x)}\)\( \frac{ k'(y)^2}{\eta(y)}\)$}=-\cos^4 (\tan(2f(x)+2k(y))).\end{align}
It follows from \e{5.38}  that at least one of $f(x),k(y)$ is a constant function. But this is impossible, since it leads to $$\cos^4 (\tan(2f(x)+2k(y)))=0.$$ Consequently, this case also cannot occur.
\end{proof}
 
 The following result is  a special case of Theorem \ref{T:5.1}.
 
 \begin{corollary}  \label{C:5.1}  Every  minimal flat  $\theta$-slant surface in ${\bf C}^2_1$ is either an open portion of a totally geodesic slant plane or congruent to the surface defined by 
  \begin{equation}\begin{aligned}\notag & \psi(x,y)=\Bigg(x+\frac{y}{2}  \cosh^2\theta+(\i -\sinh\theta ) f(y),
  \\&\hskip.1in  x-y+\frac{y}{2} \cosh^2\theta+(\i - \sinh\theta ) f(y)-\i y\sinh\theta  \Bigg)\end{aligned}\end{equation} 
for some  function $f(y)$.  
\end{corollary}

\begin{remark} When $\theta=0$, Corollary \ref{C:5.1} reduces to a result of \cite{CV}. \end{remark}

\begin{remark} If $\alpha(y)$ and $f(y)$ are functions defined on the entire real line, then the minimal flat Lorentzian surface defined in Theorem 5.1 is a complete surface.  Consequently, there exist infinitely many complete minimal flat Lorentzian surfaces in ${\bf C}^2_1$. Moreover, Corollary 5.1 shows that there exist infinitely many complete minimal flat slant surfaces in ${\bf C}^2_1$. \end{remark}

\section{Classification of minimal flat slant surfaces in $CP^2_1(4)$}

The following lemma follows easily from the proof of Theorem \ref{P:4.1}.

\begin{lemma} The only minimal flat slant surfaces in  a Lorentzian complex space form $\tilde M^2_1(4c)$ with $c\ne 0$ are the Lagrangian ones.\end{lemma}
\begin{proof} Let $M$ be a minimal flat  Lorentzian slant surface in $\tilde M^2_1(4c)$ with $c\ne 0$. Then $\alpha$ is constant. Thus \e{4.11} implies that $\sinh \alpha\cosh \alpha=0$, which is impossible unless $\alpha=0$, i.e., $M$ is Lagrangian.
\end{proof}

The following theorem completely classifies minimal flat  slant surfaces in $CP^2_1(4)$.

\begin{theorem}  \label{T:6.1} If $L:M\to CP^2_1(4)$ is a  minimal flat slant surface in the Lorentzian complex projective plane $CP^2_1(4)$, then $L$ is Lagrangian. Moreover, the immersion is congruent to $\pi\circ \tilde L$, where 
\begin{equation}\begin{aligned} \label{6.1} & \tilde L(x,y)=\text{\small$ \frac{1}{\sqrt{3}}$}\Bigg(\sqrt{2}e^{\frac{\i }{2a}(x-a^2y)}  \cosh\(\text{\small$\frac{\sqrt{3}}{2a}$}(x+a^2y)\)  ,e^{\frac{\i}{a} (a^2y-x)}, \\&\hskip1.3in\sqrt{2}e^{\frac{\i }{2a}(x-a^2y)} \sinh\(\text{\small$\frac{\sqrt{3}}{2a}$}(x+a^2y)\)\Bigg),\end{aligned}\end{equation} $a$ is a nonzero real number and $\pi:S^5_2(1)\to CP^2_1(4)$ is the Hopf fibration.
 \end{theorem}
\begin{proof}  Let $L:M\to CP^2_1(4)$ be a minimal flat  slant surface in  $CP^2_1(4)$. Then $L$ is Lagrangian according to Lemma 6.1. 

As in the proof of Theorem \ref{P:4.1},  we may assume that $M$ is  an open connected subset  of  $\mathbb E^2_1$  with
\begin{align}\label{6.2} g_o=-dx\otimes dy-dy\otimes dx.\end{align}

Let $e_1,e_2,e_3,e_4$ be as in the proof of Theorem \ref{P:4.1}. Then  we have $$\beta=\mu =\omega=\Phi=0.$$ Thus, we see from \e{4.5}, \e{4.7}, \e{4.9} and \e{4.10} that $\gamma$ and $\lambda$ are nonzero real numbers satisfying $\gamma\lambda=-1$. Hence, if we put $\lambda=-a^3$, then \e{4.2} reduces to
\begin{align} \label{6.3} & h(e_1,e_1)=\frac{Je_2}{a^3} ,\; h(e_1,e_2)=0,\;\; h(e_2,e_2)=-a^3 Je_1. 
\end{align}
Therefore, if $\tilde L:M\to S^5_2(1)$ is a horizontal lift of $L$ (cf. \cite{R}), then we have
\begin{equation}\begin{aligned} \label{6.4} \tilde L_{xx}=\frac{\i}{a^3} L_y,\;\; \tilde L_{xy}=\tilde L,\;\; \tilde L_{yy}=-\i a^3  L_x.
\end{aligned}\end{equation}
It follows from the first two equations in \e{6.4} that $$a^3\tilde L_{xxx}=\i \tilde L.$$ Solving this equation gives
\begin{equation}\begin{aligned} \label{6.5} & \tilde L=e^{\i x/(2a)}B(y)\(\cosh \(\text{\small$\frac{\sqrt{3}x}{2a}$}\) B(y)+\sinh \(\text{\small$ \frac{\sqrt{3}x}{2a}$}\)  C(y)\)   \\& \hskip.9in +e^{-\i x/a} A(y)
\end{aligned}\end{equation} for some functions $A(y),B(y),C(y)$. Substituting this into the first equation in \e{6.4} gives
\begin{align} \label{6.6} & A'(y)=\i a A(y),\\&\label{6.7} 2 B'(y)+\i a B(y)=\sqrt{3}a C(y),\\& \label{6.8}2C'(y)+\i a C(y)=\sqrt{3}a B(y).\end{align}
After solving these differential equations we have
\begin{align} \label{6.9} & A(y)=c_1 e^{iay},\\&\label{6.10}  B(y)=(b_2e^{\sqrt{3}ay}+b_3)e^{-\frac{1}{2}(\i +\sqrt{3})ay},\\& \label{6.11} C(y)=(b_2e^{\sqrt{3}ay}-b_3)e^{-\frac{1}{2}(\i +\sqrt{3})ay}
\end{align} for some constant vectors $c_1,b_2,b_3$.
Combining these with \e{6.5} gives
\begin{align} \notag & \tilde L(x,y)=e^{\frac{\i }{2a}(x-a^2y)}\left\{c_2 \cosh\(\text{\small$\frac{\sqrt{3}}{2a}$}(x+a^2y)\)  +c_3\sinh\(\text{\small$\frac{\sqrt{3}}{2a}$}(x+a^2y)\)\right\}\\&\notag \hskip1.8in 
+c_1 e^{\i (ay-\frac{x}{a})},\end{align} 
where $c_1,c_2,c_3$ are vectors in ${\bf C}^3_1$. Consequently, after choosing suitable initial conditions we obtain the immersion \e{6.1}. 
\end{proof}

\section{Minimal flat slant surfaces in $CH^2_1(-4)$}

Similarly, \ we have the following  classification of minimal flat Lagrangian surfaces in $CH^2_1(-4)$.
\begin{theorem}  \label{T:7.1}  If $L:M\to CH^2_1(-4)$ is a  minimal flat slant surfaces in  the Lorentzian complex projective plane $CH^2_1(-4)$, then $L$ is Lagrangian. Moreover, it is congruent to $\pi\circ \tilde L$, where 
\begin{equation}\begin{aligned} \label{6.12} & \tilde L(x,y)=\text{\small$\frac{1}{\sqrt{3}}$}\Bigg(\sqrt{2}e^{-\frac{\i }{2a}(x+a^2y)}  \cosh\(\text{\small$\frac{\sqrt{3}}{2a}$}(x-a^2y)\)  ,e^{\i (ay+\frac{x}{a})}, \\&\hskip1.3in\sqrt{2}e^{- \frac{\i }{2a}(x+a^2y)} \sinh\(\text{\small$\frac{\sqrt{3}}{2a}$}(x-a^2y)\)\Bigg),\end{aligned}\end{equation} $a$ is a nonzero real number and $\pi:H^5_2(-1)\to CH^2_1(-4)$ is the Hopf fibration.
 \end{theorem}
\begin{proof} This  can be proved in a way similar to the proof of Theorem \ref{T:6.1}. So, we omit the details.
\end{proof}

\begin{remark} The surfaces defined by (6.1) and  (7.1) are also complete. \end{remark}
\begin{remark}   Further results on minimal Lorentzian surfaces in Lorentzian complex space forms have been later obtained in \cite{c3} (added on May 8, 2008).
\end{remark}

\end{document}